\documentclass{amsart}
\usepackage{amsmath,amsthm}
\usepackage{amssymb,url}

\title{Normal numbers and completeness results for difference sets}
\author{Konstantinos A. Beros}
\thanks{The author acknowledges the US NSF grant DMS-0943870 for the support of his research.  He also wishes to thank Bill Mance for useful conversations on the subject of normal numbers.  Finally, the author is grateful to his brother Achilles Beros for thorough reading of various drafts of this paper, for his insightful comments and for subsequent discussions.}
\keywords{Difference hierarchy, normal numbers, completeness.}
\subjclass[2010]{%
03E15, 
28A05. 
}
\address{Department of Mathematics, University of North Texas, General Academics Building 435, 1155 Union Circle, \#311430, Denton, TX 76203-5017}
\email{beros@unt.edu}

\usepackage{hyperref,color}

\hypersetup{
pdfview=fitH,
pdfpagemode=UseNone,
colorlinks=true,
allcolors=blue,
linktocpage=true
}

\theoremstyle{plain}

\newtheorem{theorem}{Theorem}[section]

\newtheorem{corollary}[theorem]{Corollary}

\newtheorem*{theorem*}{Theorem}
\newtheorem*{corollary*}{Corollary}
\newtheorem*{proposition*}{Proposition}

\theoremstyle{definition}

\newtheorem{definition}[theorem]{Definition}
\newtheorem*{claim}{Claim}

\theoremstyle{remark}

\newtheorem*{remark*}{Remark}
\numberwithin{equation}{section}

\def\concat{{}^\smallfrown}
\def\upto{\upharpoonright}

\def\NN{\omega}
\def\ZZ{{\mathbb Z}}

\def\RR{{\mathbb R}}

\begin{document}

\begin{abstract}
We consider some natural sets of real numbers arising in ergodic theory and show that they are, respectively, complete in the classes $\mathcal D_2 (\mathbf\Pi^0_3)$ and $\mathcal D_\omega (\mathbf \Pi^0_3)$, that is, the class of sets which are 2-differences (respectively, $\omega$-differences) of $\mathbf \Pi^0_3$ sets.
\end{abstract}

\maketitle

\section{Introduction}

A recurring theme in descriptive set theory is that of analyzing the descriptive (or definable) complexity of naturally occurring sets from other areas of mathematics.  In the present work, we consider certain sets which arise in ergodic theory.  

Suppose that $T: [0,1] \rightarrow [0,1]$ is a Borel map which preserves Lebesgue measure.  It is of interest to consider those points $x \in [0,1]$ which exhibit ``chaotic'' or random behavior with respect to $T$ and its iterates.  For instance, one might consider those points $x$ such that $\{ T^n (x) : n \in \NN\}$ is dense in $[0,1]$.  (Here $T^n (x)$ denotes the $n$-fold iterate of the map $T$, applied to $x$.)  Given such a $T$, there is another type of chaotic behavior that is related to uniform distribution.  Specifically, one considers those points $x$ such that the sequence $x , T(x) , T^2(x) , \ldots$ is {\em uniformly distributed}, that is, for each subinterval $I \subseteq [0,1]$, 
\[
\lim_{n \rightarrow \infty}\frac{\mathrm{Cardinality}( \{ i < n : T^i (x) \in I \} )}{n} = \mathrm{length} (I).
\]
If one considers the transformation $T(x) = bx \, (\mathrm{mod}\ 1)$, for some fixed integer $b \geq 2$, then an $x \in [0,1]$ exhibiting this type of chaotic behavior is called {\em normal to base $b$}.  It can be shown that $x \in [0,1]$ is normal to base $b$ if, for each integer $k \geq 1$ and $m < b^k$, 
\[
\lim_{n \rightarrow \infty}\frac{\mathrm{Cardinality}( \{ i < n : m/b^k \leq T^i (x) < (m+1)/b^k \} )}{n} = 1/b^k.
\]
In turn, this is equivalent to the combinatorial statement that every finite string, $\sigma$, of numbers $0$ through $b-1$ occurs in the $b$-ary expansion of $x$, with frequency (in the limit) $b^{-\mathrm{length} (\sigma)}$.  It is a consequence of the Birkhoff Ergodic Theorem that, for each integer $b \geq 2$, the set of $x \in [0,1]$ which are normal to base $b$ has Lebesgue measure $1$.

Next we introduce some notions from descriptive set theory.  Recall that a {\em pointclass} is a family of sets which can be described in any complete separable metric space, e.g., the classes of $F_\sigma$, $G_\delta$, or analytic sets are all pointclasses.  In the present work, the pointclasses we consider are $\mathbf \Pi^0_3$, $\mathcal D_2 (\mathbf \Pi^0_3)$ and $\mathcal D_\omega (\mathbf \Pi^0_3)$.  The class $\mathbf \Pi^0_3$ is that of sets which have the form $\bigcap_{m\in \NN} \bigcup_{n\in \NN} F_{m,n}$, with each $F_{m,n}$ a closed set.  The class $\mathcal D_2 (\mathbf \Pi^0_3)$ is that of sets having the form $A\setminus B$, with $A,B \in \mathbf \Pi^0_3$.  Finally, sets in $\mathcal D_\omega (\mathbf \Pi^0_3)$ have the form $\bigcup_k A_{2k+1} \setminus A_{2k+2}$, where $A_1 \supseteq A_2 \supseteq \ldots$ are $\mathbf \Pi^0_3$ sets.

\begin{definition}
If $\Gamma$ is a pointclass, we say that $X \in \Gamma$ is {\em $\Gamma$-complete} iff every $Y \subseteq \{0,1\}^\NN$, with $Y \in \Gamma$, is a continuous preimage of $X$.
\end{definition}

\noindent In a sense, a $\Gamma$-complete set ``encodes'' all $\Gamma$ subsets of $\{0,1\}^\NN$.  

There are many well-known $\mathbf \Pi^0_3$- complete sets.  For instance, the set
\[
\{ x \in \NN^\NN : \lim_{n \rightarrow \infty} x(n) = \infty\}.
\]
(See \S23 of Kechris \cite{KECHRIS.dst}.)  In 1994, Haseo Ki and Tom Linton \cite{ki.linton} published an interesting completeness result related to the set of numbers normal to base $b$.

\begin{theorem}[Ki-Linton]
The set of real numbers which are normal to base $b$ is $\mathbf \Pi^0_3$-complete.
\end{theorem}

More work in this vein has been done subsequently by others.  For instance, Ver\'onica Becher, Pablo Heiber and Ted Slaman \cite{becher.heiber.slaman.2} showed that the set of real numbers which are normal to all bases is also $\mathbf \Pi^0_3$-complete.  In other work, Becher and Slaman \cite{becher.slaman.sigma4} have shown that the set of numbers normal to at least one base is $\mathbf \Sigma^0_4$-complete.  

In general, however, there are not many known completeness results for difference classes, e.g., $\mathcal D_2 (\mathbf \Pi^0_3)$ and $\mathcal D_\omega (\mathbf \Pi^0_3)$.  In what follows, we shall prove completeness results for the classes $\mathcal D_2 (\mathbf \Pi^0_3)$ and $\mathcal D_\omega (\mathbf \Pi^0_3)$.  Before stating our result, we introduce some more terminology.

For the present work, we mostly restrict attention to the case of $b=2$, as this will simplify our notation somewhat.  As a weakened form of normality, one may consider those $x \in [0,1]$ which are {\em order-$k$ normal} in base 2.  That is, such that, for each $j < 2^k$,
\[
\lim_{n \rightarrow \infty}\frac{\mathrm{Cardinality}( \{ i < n : T^i(x) \in [j/2^k , (j+1)/2^k) \} )}{n} = 2^{-k},
\]
where $T:[0,1] \rightarrow [0,1]$ is the map $x \mapsto 2x \, (\mathrm{mod} \ 1)$.  Let $N_k$ denote the set of numbers which are order-$k$ normal in base $2$.  Note that a real number $x$ is normal iff it is order-$k$ normal, for each $k \geq 1$.


Examining the proofs in Ki-Linton \cite{ki.linton}, one may deduce the following theorem.

\begin{theorem}[Ki-Linton]
The sets $N_1$ and $N_2$ are $\mathbf \Pi^0_3$-complete.
\end{theorem}

Inspired by this observation, we proved the following result.

\begin{theorem}\label{T1}
The set $N_1 \setminus N_2$ is $\mathcal D_2 (\mathbf \Pi^0_3)$-complete.
\end{theorem}

\begin{corollary}
The set $N_1 \setminus N_2$ is properly $\mathcal D_2 (\mathbf \Pi^0_3)$.
\end{corollary}

The method of our proof is somewhat different from that of Ki-Linton.  Specifically, we employ a permitting structure to construct a reduction of an arbitrary $\mathcal D_2 (\mathbf \Pi^0_3)$ set to $N_1 \setminus N_2$.  Our task is necessarily complicated by the fact that there are not many combinatorially tractable sets which are known to be $\mathcal D_2 (\mathbf \Pi^0_3)$-complete.

In response to a question posed, in conversation, by Su Gao, we extended the method used to prove Theorem~\ref{T1} and obtained the following result.

\begin{theorem}\label{T2}
The set $\bigcup_k N_{2k+1} \setminus N_{2k+2}$ is $\mathcal D_\omega (\mathbf \Sigma^0_3)$-complete.
\end{theorem}


\section{Preliminaries and notation}

We now introduce some notation which largely follows Kechris \cite{KECHRIS.dst}, our principal reference for descriptive set theory.

Let $\langle \cdot , \cdot \rangle : \NN^2 \rightarrow \NN$ be a fixed bijective pairing function such that, for fixed $m \in \NN$, the sequence $\langle m , 0\rangle , \langle m , 1 \rangle , \ldots$ is increasing.  Likewise, we fix a bijective ``triple function'' $\langle \cdot , \cdot , \cdot \rangle : \NN^3 \rightarrow \NN$.

Let $\{0,1\}^n$ denote the set of finite binary strings of length $n$, $\{0,1\}^{\leq n}$ denote the set of binary strings of length not greater than $n$, and $\{0,1\}^{<\omega}$ denote the set of all finite binary strings (of all lengths).  Let $\{0,1\}^\NN$ denote the set of all infinite binary sequences, equipped with the product, over $\NN$, of the discrete topology on $\{0,1\}$.  For $x \in \{0,1\}^\NN$, let $x(n)$ denote the $n$th term of $x$ and let $x \upto n$ denote the finite string $(x(0) , \ldots , x(n-1))$.  For $\sigma \in \{0,1\}^{< \omega}$, let $[\sigma]$ denote the basic open set \[
\{ x \in \{0,1\}^\NN : \sigma \mbox{ is an initial segment of } x\}.
\]
For $\sigma , \tau \in \{0,1\}^{< \omega}$, let $\sigma \concat \tau$ denote the concatenation of $\sigma$ and $\tau$.  For $\alpha \in \{0,1\}^{< \omega}$, let $\alpha^n$ denote the $n$-fold concatenation of $\alpha$ with itself.  Similarly, let $\alpha^\infty$ denote the infinite binary sequence $\alpha \concat \alpha \concat \ldots$.  For $\sigma \in \{0,1\}^{< \omega}$, let $|\sigma|$ denote the length of $\sigma$.

For $a \in \ZZ$ and $x \in \{0,1\}^\NN$, let $a \, . \, x$ denote the number
\[
a + \sum_{n = 0}^\infty x(n) / 2^{n+1}.
\]
That is, $a \, . \, x$ is the least real number greater than or equal to $a$ whose fractional part has binary expansion $x$.

If $\alpha$ and $\sigma$ are finite binary strings, with $|\alpha \| \leq |\sigma|$, let $d_\alpha (\sigma)$ be
\[
\frac{\mathrm{Cardinality}\big(\{ i <  |\sigma| - |\alpha| : (\exists \beta \in \{0,1\}^{< \omega}) (\sigma \upto (i + |\alpha|) = \beta \concat \alpha)\}\big)}{|\sigma|}.
\]
In other words, $d_\alpha (\sigma)$ indicates the proportion of substrings of $\sigma$ which are equal to $\alpha$.

For the rest of this paper, we will use the following well-known equivalent definition of order-$k$ normality in base 2, rather than that introduced in the previous section.

\begin{definition}
A real number, $a \, . \, x$, is {\em order-$k$ normal} in base 2 iff, for each $\alpha \in \{0,1\}^k$, the sequence $(d_\alpha (x\upto s))_{s \in \NN}$ is convergent, with
\[
\lim_{s \rightarrow \infty} d_\alpha (x \upto s) = 2^{-k}.
\]
We let $N_k$ denote the set of all order-$k$ normal numbers in $[0,1]$.
\end{definition}


Proving that this definition is equivalent to the one given earlier is a relatively straightforward matter. (See Kuipers-Niederreiter \cite{kuipers.niederreiter}, Chapter 1, exercise 8.7.)


\section{The proof of Theorem~\ref{T1}}

Let $L = \bigcap_m \bigcup_n L_{m,n}$ and $F = \bigcap_m \bigcup_n F_{m,n}$ be fixed $\mathbf \Pi^0_3$ subsets of $\{0,1\}^\NN$, with the $L_{m,n}$ and $F_{m,n}$ all closed sets.  With no loss of generality, we may assume that, for each $m$,
\[
L_{m,0} \subseteq L_{m,1} \subseteq \ldots \qquad \mbox{and} \qquad F_{m,0} \subseteq F_{m,1}\subseteq \ldots.
\]
Also, since we will be considering the difference set $L\setminus F$, we may assume that $L \supseteq F$.  Were this not so, we could replace $F$ with $L \cap F$.  We now proceed to define a continuous map $f : \{0,1\}^\NN \rightarrow \mathbb R$ such that $f^{-1} (N_1 \setminus N_2) = L\setminus F$.

In the first place, let 
\[
\alpha_n = (0110)^n \concat (10)
\]
and
\[
\beta_n = (0110)^n \concat 0.
\]
Observe that 
\[
\lim_{k \rightarrow \infty} d_{10} (\alpha_n^\infty \upto k) = (n+1)/(4n+2) > 1/4
\]
and
\[
\lim_{k \rightarrow \infty} d_0 (\beta_n^\infty \upto k) = (2n+1)/(4n+1) > 1/2.
\]
Also, if $y \in \RR$ is of the form $0\, . \, \alpha_{i_0}^{a_0} \concat \beta_{j_0}^{b_0} \concat \alpha_{i_1}^{a_1} \concat \beta_{j_1}^{b_1} \concat \ldots$, then $y \in N_1$ if $j_p \rightarrow \infty$, as $p \rightarrow \infty$, since the $\alpha$'s do not affect the density of 0's and 1's in the binary expansion of $y$.  In addition, if  both $i_p , j_p \rightarrow \infty$, as $p \rightarrow \infty$, then $y \in N_2$.  This follows from the fact that the real number $0.01100110\ldots$ is order-2 normal and inserting a density-zero set of digits does not affect normality.

Given $x \in \{0,1\}^\NN$, we will let 
\[
f(x) = 0\, . \, \alpha_{i_0}^{a_0} \concat \beta_{j_0}^{b_0} \concat\alpha_{i_1}^{a_1} \concat \beta_{j_1}^{b_1} \concat \ldots,
\]
where $i_p$, $j_p$, $a_p$ and $b_p$ are natural numbers, defined as follows, for each $p \in \NN$:
\[
i_p = \begin{cases}
i_{p-1} + 1 & \mbox{if } p = \langle m,n \rangle, \ (\forall n' < n)([x \upto \langle m, n-1 \rangle] \cap F_{m,n'} \neq \emptyset \\ &\implies [x \upto p] \cap F_{m,n'} \neq \emptyset) \mbox{ and } [x \upto p] \cap F_{m,n} \neq \emptyset,\\

m &\mbox{if } p = \langle m,n \rangle \mbox{ and } ((\exists n' < n)([x \upto \langle m,n-1 \rangle] \cap F_{m,n'} \neq \emptyset \\ 
&\mbox{and } [x \upto p] \cap F_{m,n'} = \emptyset) \mbox{ or } [x \upto p] \cap F_{m,n} = \emptyset)
\end{cases}
\] 
We refer to the two cases in this definition as {\em Case 1} and {\em Case 2}.  The definition of $j_p$ is identical to that of $i_p$, except with the $L_{m,n}$ in place of the $F_{m,n}$.  We let
\[
a_p = \begin{cases}
1 &\mbox{if $p = \langle m,n \rangle$ and Case 1 from the definition of $i_p$ holds,}\\
k &\mbox{if $p = \langle m,n \rangle$ and Case 2 from the definition of $i_p$ holds,}\\ &\mbox{where $k$ is chosen to be large enough that}\\
&d_{10} (\alpha_{i_0}^{a_0} \concat \beta_{j_0}^{b_0} \concat\alpha_{i_1}^{a_1} \concat \ldots \concat \beta_{j_{p-1}}^{b_{p-1}} \concat \alpha_{i_p}^k) \geq (m+(3/4))/(4m+2).
\end{cases}
\]
We may alway find such a $k$ in Case 2, since $i_p = m$ and, therefore,
\[
\lim_{k \rightarrow \infty} d_{10} (\sigma \concat \alpha_m^\infty \upto k) = (m+1)/(4m+2),
\]
for any $\sigma \in \{0,1\}^{< \omega}$.  Similarly, we define
\[
b_p = \begin{cases}
1 &\mbox{if $p = \langle m,n \rangle$ and Case 1 from the definition of $j_p$ holds,}\\
k &\mbox{if $p = \langle m,n \rangle$ and Case 2 from the definition of $j_p$ holds,}\\ &\mbox{where $k$ is chosen to be large enough that}\\
&d_{0} (\alpha_{i_0}^{a_0} \concat \beta_{j_0}^{b_0} \concat\alpha_{i_1}^{a_1} \concat \ldots \concat \beta_{j_{p-1}}^{b_{p-1}} \concat \alpha_{i_p}^{a_p} \concat \beta_{j_p}^{k}) \geq (2m+(3/4))/(4m+1).
\end{cases}
\]
Again, we may always find such a $k$, since $j_p = m$ and, hence,
\[
\lim_{k \rightarrow \infty} d_0 (\sigma\concat \beta_m^\infty \upto k) = (2m+1)/(4m+1),
\]
for any $\sigma \in \{0,1\}^{< \omega}$.

\medskip

\noindent{\scshape Verification.}  The claims below will complete the proof.  In particular, they will establish that $f^{-1} (N_1 \setminus N_2) = L\setminus F$.

\begin{claim}
The map $f : \{0,1\}^\NN \rightarrow \RR$ is continuous.
\end{claim}

\begin{proof}[Proof of claim]
The continuity of $f$ follows from the observation that $i_0 , \ldots , i_p$, \newline $j_0 , \ldots , j_p$, $a_0 , \ldots , a_p$ and $b_0 , \ldots , b_p$ are all determined by the first $p$ terms of $x$.  In particular, at least the first $p$ digits of the binary expansion of $f(x)$ are determined by the first $p$ terms of $x$.  Thus, if $x , y \in \{0,1\}^\NN$ are such that $x \upto p = y \upto p$,
\[
|f(x) - f(y)| \leq 2^{-p}.
\]
Thus, $f$ is continuous.
\end{proof}

At this juncture, we introduce some convenient terminology.  For fixed $x \in \{0,1\}^\NN$, we say that $m \in \NN$ {\em acts infinitely often} for the $i_p$ (resp., $j_p$) iff there exist infinitely many $n$ such that Case 2 holds for $p = \langle m,n \rangle$, in the definition of $i_p$ (resp., $j_p$).  Otherwise, we say that $m$ {\em acts finitely often} for the $i_p$ (resp., $j_p$).  Also, note that $i_p \rightarrow \infty$, as $p \rightarrow \infty$, iff each $m \in \NN$ acts only finitely often for the $i_p$.  Likewise, for the $j_p$.  Although this terminology refers implicitly to a specific $x$, we will supress mention of $x$, since $x$ will always be fixed in what follows.

\begin{claim}
If $x \in F$, then $f(x) \in N_2$.
\end{claim}

\begin{proof}[Proof of claim]
As noted above, it will suffice to show that $i_p \rightarrow \infty$ and $j_p \rightarrow \infty$.  In turn, it will suffice to show that each $m$ acts only finitely many times for both the $i_p$ and the $j_p$.  Indeed, suppose $x \in F$ and fix $m \in \NN$.  Let $n_0$ be such that that $x \in F_{m,n_0}$.  It follows that $x \in F_{m,n}$, for each $n \geq n_0$, since $F_{m,0} \subseteq F_{m,1} \subseteq \ldots$.  Hence, $[x \upto p] \cap F_{m,n} \neq \emptyset$, for all $p \geq \langle m,n_0\rangle$ and $n \geq n_0$.  We may also assume $n_0$ is large enough that, if $n \geq n_0$ and $n' < n_0$, we have $[x \upto \langle m , n-1 \rangle] \cap F_{m,n'} = \emptyset$.  In particular, we are in Case 1 of the definition of $i_p$, provided $p = \langle m,n \rangle$, with $n \geq n_0$.  Thus, $m$ acts only finitely many times for the $i_p$.

We omit the corresponding argument for the $j_p$, as it is entirely analogous, using the fact that $x \in L \subseteq F$.  This completes the proof of the claim.
\end{proof}

\begin{claim}
If $x \in L \setminus F$, then $f(x) \in N_1 \setminus N_2$.
\end{claim}

\begin{proof}[Proof of claim]
Fix $x \in L\setminus F$.  As in the previous claim, each $m$ acts only finitely many times for the $j_p$.  It follows that $j_p \rightarrow \infty$ and hence $f(x) \in N_1$.  On the other hand, we shall see that some $m$ acts infinitely often for the $i_p$.  Indeed, since $x \notin F$, there exists $m$ such that $x \notin F_{m,n}$, for all $n$.  For each $n$, let $k_n$ be least such that $[x \upto k_n] \cap F_{m,n} = \emptyset$.  Note that $k_0 \leq k_1 \leq \ldots$, since $F_{m,0} \subseteq F_{m,1} \subseteq \ldots$.

We consider two cases.  In the first instance, suppose that there are infinitely many $r$ such that $k_r > \langle m,r \rangle$.  We may therefore select $r_0 < r_1 < \ldots$ and $n_0 < n_1 < \ldots$ such that
\begin{itemize}
\item $(\forall e) (r_e < n_e)$ and
\item $(\forall e) (\langle m , n_e -1 \rangle < k_{r_e} \leq \langle m, n_e \rangle$.
\end{itemize}
Thus, for each $p = \langle m , n_e \rangle$, there will be an $n' < n_e$ (namely $n' = r_e$) such that $[x \upto \langle m, n_e -1 \rangle] \cap F_{m,n'} \neq \emptyset$ (since $\langle m,n_e -1 \rangle < k_{r_e}$), but $[x \upto \langle m, n_e \rangle] \cap F_{m,n'} = \emptyset$ (since $\langle m,n_e\rangle \geq k_{r_e}$).  It follows that, for each $p = \langle m,n_e\rangle$, we will be in Case 2 of the definition of $i_p$.


In the second case, we assume that $k_n \leq \langle m,n \rangle$, for all but finitely many $n$.  Thus, by the definition of the $k_n$, we have that $[x \upto \langle m,n\rangle ] \cap F_{m,n} = \emptyset$, for all but finitely many $n$.  Hence, for cofinitely many $n$, if $p = \langle m , n \rangle$, we are in Case 2 of the definition of $i_p$.

It now follows that $f(x) \notin N_2$, since each time $p = \langle m,n \rangle$ is in Case 2 of the definition of $i_p$, we have 
\[
d_{10} (\alpha_{i_0}^{a_0} \concat \beta_{j_0}^{b_0} \concat\alpha_{i_1}^{a_1} \concat \ldots \concat \beta_{j_{p-1}}^{b_{p-1}} \concat \alpha_{i_p}^{a_p}) \geq (m+(3/4))/(4m+2) > 1/4,
\]
and this occurs infinitely often for some fixed $m$.
\end{proof}

\begin{claim}
If $x \notin L$, then $f(x) \notin N_1$.
\end{claim}

\begin{proof}[Proof of claim]
As in the second half of the proof of the previous claim, we observe that there is some $m$ which acts infinitely often for the $j_p$ and conclude that $f(x) \notin N_1$.
\end{proof}

This completes the proof.


\section{A generalization}

We now indicate how to generalize the preceeding argument to an arbitrary fixed base $b$.  In the first place, our definition of $d_\sigma (\alpha)$ may be extended to strings $\sigma , \alpha \in \{0,1, \ldots , b-1\}^{< \omega}$.  Namely, $d_\sigma (\alpha)$ is the number of times $\sigma$ occurs as a substring of $\alpha$, divided by $|\alpha|$.  Also, if $x \in \{0,1, \ldots , b-1\}^\NN$, we let 
\[
(a \, . \, x)_b = a + \sum_{n=0}^\infty \frac{x(n)}{b^{n+1}}.
\]

\begin{definition}
For integers $b, r$, with $b \geq 2$ and $r \geq 1$, and $x \in \{0 , 1 , \ldots , b-1\}^\NN$, we say that a real number $(a \, . \, x)_b$ is {\em order-$r$ normal} in base $b$ iff, for each $\sigma \in \{0 , 1 , \ldots , b-1\}^r$, 
\[
\lim_{k \rightarrow \infty} d_\sigma (x \upto k) = b^{-r}.
\]
We let $N^b_r$ denote the set of real numbers which are order-$r$ normal in base $b$.
\end{definition}

It is important to note that $N^b_s \subseteq N^b_r$, if $r < s$.  This follows from the fact that, for any $\sigma \in \{ 0, 1 , \ldots , b-1\}^r$, there are $b^{r-s}$ many $\tau \in \{ 0, 1 , \ldots , b-1\}^s$, having $\sigma$ as an initial segment.  Hence, if $(0\, . \, x)_b \in N^b_s$ and $\sigma \in \{ 0, 1 , \ldots , b-1\}^r$,
\[
\lim_{k \rightarrow \infty} d_\sigma (x \upto k) = b^s \cdot b^{r-s} = b^r.
\]
We now state and sketch the proof of a generalization of Theorem~\ref{T1}.

\begin{theorem}
For each base $b \geq 2$ and $s > r \geq 1$, the set $N^b_r \setminus N^b_s$ is $\mathcal D_2 (\mathbf \Pi^0_3)$-complete.
\end{theorem}

\begin{proof}[Sketch of proof]
I.~J.~Good \cite{good.normal.recurring} showed that, for each $b,r$ as in the definition above, there exists a finite string $\theta \in \{0,1, \ldots , b-1\}^{b^r}$ such that the real number $(0 \, . \, \theta ^\infty)_b \in N^b_r$.  If $s > r$, there are $b^s$ possible strings of length $s$ using digits $\{ 0, 1, \ldots , b-1\}$.  Thus, if $\theta \in \{0,1, \ldots , b-1\}^{b^r}$, then $(0 \, . \, \theta^\infty)_b$ cannot be order-$s$ normal in base $b$, since there are at most $b^r$ substrings of $\theta^\infty$ of any fixed length.  It follows that, if $\theta$ is as in Good's result with $(0\, . \, \theta^\infty)_b \in N^b_r$, then $(0\, . \, \theta^\infty)_b$ is not order-$s$ normal, for any $s > r$.

Now fix a base $b$ and $r<s$.  Let $\theta , \mu \in \{ 0 ,1, \ldots , b-1\}^{< \omega}$ be such that 
\begin{itemize}
\item $|\theta| = b^s$,
\item $|\mu | = b^r$, 
\item $(0\, . \, \theta^\infty)_b \in N^b_s$ and
\item $(0\, . \, \mu^\infty)_b \in N^b_r$.
\end{itemize}
Following the notation of the proof of Theorem~\ref{T1}, let $\alpha_n = \theta^n \concat \mu$ and $\beta_n = \theta^n \concat 0$.  Note that, for all $n$, $0 \, . \, \alpha_n^\infty \in N^b_r \setminus N^b_s$, whereas, for all $n$, $0\, . \, \beta_n^\infty \notin N^b_r$.  Observe that, if 
\[
y = (0\, . \, \alpha_{i_0}^{a_0} \concat \beta_{j_0}^{b_0} \concat\alpha_{i_1}^{a_1} \concat \beta_{j_1}^{b_1} \concat \ldots)_b,
\]
then $y \in N^b_r$ if $j_p \rightarrow \infty$, as $p \rightarrow \infty$.  Similarly, $y \in N^b_s$, if $i_p \rightarrow \infty$ and $j_p \rightarrow \infty$.

Following the proof of Theorem~\ref{T1}, with these new $\alpha_n$ and $\beta_n$ and certain other minor modifications yields a proof of the theorem above.
\end{proof}



\section{The proof of Theorem~\ref{T2}}

Fix a descending sequence $F_1 \supseteq F_2 \supseteq \ldots$ of $\mathbf \Pi^0_3$ sets.  For each $k$, let $F_{k,m,n}$ be closed sets with 
\[
F_k = \bigcap_m \bigcup_n F_{k,m,n}.
\]
We may assume that, for each pair $k,m$, we have $F_{k,m,0} \subseteq F_{k,m,1} \subseteq \ldots$.  Our objective is to show that $\bigcup_k F_{2k+1} \setminus F_{2k+2}$ is a continuous preimage of $\bigcup_k N_{2k+1} \setminus N_{2k+2}$, where $N_k$ denotes the set of real numbers in $[0,1]$ which are order-$k$ normal.  To this end, we will define a continuous function $f: \{0,1\}^\NN \rightarrow \{0,1\}^\NN$ such that, for each $x \in \{0,1\}^\NN$, 
\[
x \in F_k \iff 0 \, . \, f(x) \in N_k.
\]
Given $x \in \{0,1\}^\NN$, we will define finite binary strings, $\sigma_t$, with $\sigma_0 \preceq \sigma_1 \preceq \ldots$ and let $f(x) = \bigcup \sigma_t$.

Before proceeding, we introduce some notation for the sake of the construction.  For each $i \in \NN$, $i > 0$, let $\eta_i \in \{0,1\}^i$ be, as in I.~J.~Good \cite{good.normal.recurring}, such that $0\, . \, (\eta_i)^\infty$ is order-$k$ normal.  Note that each $\alpha \in \{0,1\}^i$ must occur exactly once in each period of the repeating decimal $0 \, . \, (\eta_i)^\infty$.  Also, since $|\eta_i| = 2^i$, the real number $0 \, . \, (\eta_i)^\infty$ is not order-$(i+1)$ normal, as there are at most $2^i$ distinct substrings of $(\eta_i)^\infty$ of any fixed length.  For each $i$, we therefore fix an $\alpha_i \in \{0,1\}^i$ which is not a substring of $(\eta_{i-1})^\infty$.

For each triple $k,m,n$, we now let 
\[
\tau_{k,m,n} = (\eta_{k,m})^i \concat (\eta_{k-1})^j,
\]
where $i,j \in \NN$ are chosen such that the following hold.
\begin{itemize}
\item For each triple $k,m,n$ and each $\alpha \in \{0,1\}^{\leq k+m}$,
\[
\left| \left(\lim_{s\rightarrow \infty} d_\alpha \big( (\tau_{k,m,n})^\infty \upto s\big)\right) - 2^{-|\alpha|} \right| < 2^{-(k+m)}.
\]
\item For each pair $k,m$, there exists $r_{k,m} < 2^{-k}$ such that, for all $n$, 
\[
\lim_{s \rightarrow \infty} d_{\alpha_k} \big( (\tau_{k,m,n} )^\infty \upto s \big) < r_{k,m}.
\]
\item For each triple $k,m,n$ and each $\alpha \in \{0,1\}^{\leq k-1}$,
\[
\left| \left( \lim_{s \rightarrow \infty} d_\alpha \big( (\tau_{k,m,n})^\infty \upto s\big) \right)  - 2^{-|\alpha|} \right| < 2^{-\langle k,m,n \rangle}.
\]
\end{itemize}

\medskip

{\scshape The construction.}  At this point, fix $x \in \{0,1\}^\NN$.  As indicated above, we will define binary strings $\sigma_t$, determined by $x$. For each $t = \langle k,m,n \rangle$, we distinguish between two distinct cases.  We say that $t = \langle k,m,n \rangle$ is in {\em case 1} if 
\begin{itemize}
\item $[x \upto \langle m,n\rangle ] \cap F_{k,m,n} \neq \emptyset$ and,
\item for each $n' < n$, if $[x \upto \langle m,n-1\rangle] \cap F_{k,m,n'} \neq \emptyset$, then $[x \upto \langle m,n\rangle ] \cap F_{k,m,n'} \neq \emptyset$.
\end{itemize} 
Likewise, we say that $t = \langle k,m,n \rangle$ is in {\em case 2} if
\begin{itemize}
\item $[x \upto \langle m,n\rangle ] \cap F_{k,m,n} = \emptyset$ or
\item there exists $n'< n$ such that $[x \upto \langle m,n-1\rangle ] \cap F_{k,m,n'} \neq \emptyset$, but $[x \upto \langle m,n\rangle ] \cap F_{k,m,n'} = \emptyset$.
\end{itemize}

In the process of defining the binary strings $\sigma_t$, we also define binary sequences $y_t \in \{0,1\}^\NN$ such that $\sigma_t \prec y_t$ and functions $\mu_t : \{0,1\}^{< \omega} \times \NN \rightarrow \NN$ such that, for each $\alpha \in \{0,1\}^{<\omega}$ and $p \in \NN$, $\mu_t (\alpha, p)$ is the least $q \in \NN$ with
\[
\left| d_\alpha (y_t \upto q') - \left(\lim_{s \rightarrow \infty}d_\alpha (y_t \upto s)\right) \right| < 2^{-p},
\]
for all $q' \geq q$.  Note that the limit in the expression above is guaranteed to exist because $y_t$ is eventually periodic.  We call the map $\mu_t$ the {\em modulus of distribution} for $y_t$.

Suppose that $\sigma_{t-1}$ is given, we show how to define $\sigma_t$.  (In the case of $t = 0$, we let $\sigma_{-1}$ be the empty string, for notational purposes.)  

First suppose that $t = \langle k,m,n \rangle$ is in case 1.  Let $y_t = \sigma_{t-1} \concat (\eta_t)^\infty$ and $\sigma_t = \sigma_{t-1}\concat (\eta_t)^i$, where $i$ is large enough that the following hold.
\begin{enumerate}
\item For all $\alpha \in \{0,1\}^{\leq t}$,
\[
\left| d_\alpha \big( \sigma_t \big) - 2^{-|\alpha|} \right| < 2^{-t}.
\]
\item[(2a)] If $t+1$ is in case 1 and $\mu : \{0,1\}^{< \omega} \times \NN \rightarrow \NN$ is the modulus of distribution for $\sigma_t \concat (\eta_t)^i \concat (\eta_{t+1})^\infty$, then
\[
\mu \upto \{0,1\}^{\leq t} \times \{ 0 , \ldots , t\} = \mu_t \upto \{0,1\}^{\leq t} \times \{ 0 , \ldots , t\}.
\]
\item[(2b)]  If $t+1 = \langle k',m',n' \rangle$ is in case 2, $\mu : \{0,1\}^{< \omega} \times \NN \rightarrow \NN$ is the modulus of distribution for $\sigma_t \concat (\eta_t)^i \concat (\tau_{k',m',n'})^\infty$ and $p = \min \{ t , k' + m' \}$, then 
\[
\mu \upto \{0,1\}^{\leq p} \times \{ 0 , \ldots , p\} = \mu_t \upto \{0,1\}^{\leq p} \times \{ 0 , \ldots , p\}
\]
and, if $k^* = \min \{ t , k' - 1\}$,
\[
\mu \upto \{0,1\}^{\leq k^*} \times \{ 0 , \ldots , t\} = \mu_t \upto \{0,1\}^{\leq k^*} \times \{ 0 , \ldots , t\}.
\]
\end{enumerate}

Now suppose that $t = \langle k,m,n \rangle$ is in case 2.  Let $y_t = \sigma_{t-1} \concat (\tau_{k,m,n})^\infty$ and $\sigma_t = \sigma_{t-1} \concat (\tau_{k,m,n})^i$, where $i$ is large enough that the following hold.
\begin{enumerate}\setcounter{enumi}{2}
\item For each $\alpha \in \{0,1\}^{\leq k+m}$, 
\[
\left| d_\alpha (\sigma_t) - 2^{-|\alpha|} \right| < 2^{-(k+m)}.
\]
\item For each $\alpha \in \{0,1\}^{\leq k-1}$,
\[
\left| d_\alpha (\sigma_t) - 2^{-|\alpha|} \right| < 2^{-t}.
\]
\item $d_{\alpha_k} ( \sigma_t ) < r_{k,m}$, where $\alpha_k$ and $r_{k,m}$ are as above.
\item[(6a)]  If $t+1$ is in case 1, $\mu : \{0,1\}^{< \omega} \times \NN \rightarrow \NN$ is the modulus of distribution for $\sigma_t \concat (\tau_{k,m,n})^i \concat (\eta_{t+1})^\infty$ and $p = \min \{ k+m , t+1 \}$, then 
\[
\mu \upto \{0,1\}^{\leq p} \times \{ 0 , \ldots , p\} = \mu_t \upto \{0,1\}^{\leq p} \times \{ 0 , \ldots , p\}
\]
and, if $k^* = \min \{ k-1 , t+1 \}$,
\[
\mu \upto \{0,1\}^{\leq k^*} \times \{ 0 , \ldots , t\} = \mu_t \upto \{0,1\}^{\leq k^*} \times \{ 0 , \ldots , t\}.
\]
\item[(6b)]  If $t+1 = \langle k',m',n' \rangle$ is in case 2, $\mu : \{0,1\}^{< \omega} \times \NN \rightarrow \NN$ is the modulus of distribution for $\sigma_t \concat (\tau_{k,m,n})^i \concat (\tau_{k',m',n'})^\infty$ and $p = \min \{ k+m , k'+m' \}$, then 
\[
\mu \upto \{0,1\}^{\leq p} \times \{ 0 , \ldots , p\} = \mu_t \upto \{0,1\}^{\leq p} \times \{ 0 , \ldots , p\}
\]
and, if $k^* = \min \{ k-1 , k'-1\}$, 
\[
\mu \upto \{0,1\}^{\leq k^*} \times \{ 0 , \ldots , t\} = \mu_t \upto \{0,1\}^{\leq k^*} \times \{ 0 , \ldots , t\}.
\]
\end{enumerate}

We now let $f(x) = \bigcup_t \sigma_t$.  This completes the definition of $f$.

\medskip

{\scshape Verification.}  The claims below will complete the proof of Theorem~\ref{T2}.

\begin{claim}
The map $f:\{0,1\}^\NN \rightarrow \{0,1\}^\NN$ in continuous.
\end{claim}

\begin{proof}[Proof of claim]
This follows from the fact that, given $x \in \{0,1\}^\NN$, each bit of $f(x)$ is determined by finitely many bits of $x$.
\end{proof}

In what follows, let $x \in \{0,1\}^\NN$ be fixed and let $\sigma_t$, $y_t$, $\mu_t$, etc.~be defined as above for $x$.

\begin{claim}
If $x \in F_{k_0}$, then $\lim_{t\rightarrow \infty} \mu_t (\alpha , p)$ exists, for each $\alpha \in \{0,1\}^{\leq k_0}$ and $p \in \NN$.
\end{claim}

\begin{proof}[Proof of claim]
Assume $x \in F_{k_0}$.  Fix $p \in \NN$ and let $t_0 \geq \max\{k_0,p\}$ be large enough that, for all $t = \langle k,m,n \rangle \geq t_0$, whenever $k\leq k_0$ and $t$ is in case 2, we have $k+m \geq \max \{ k_0 , p \}$.  To see that there is such a $t_0$, observe that, given a fixed pair $k,m$, with $k \leq k_0$, we have $x \in F_{k,m,n}$, for all but finitely many $n$, say $n_0$ is the least such $n$.  Hence, we have that $\langle k,m,n \rangle$ is in case 1 for all $n \geq n_0$ large enough that 
\[
n' < n_0 \implies [x \upto \langle m,n \rangle] \cap F_{k,m,n'} = \emptyset.
\]
Hence, given any pair $k,m$, with $k \leq k_0$, there are only finitely many $n$ such that $\langle k,m,n \rangle$ is in case 2.  Thus, there are only finitely many $\langle k,m,n \rangle$ in case 2, with $k \leq k_0$ and $k+m <  \max \{ k_0 , p \}$.

We check that $\mu_{t+1} (\alpha , p) = \mu_t (\alpha , p)$, for all $t \geq t_0$ and $\alpha \in \{0,1\}^{\leq k_0}$.  We then conclude, by induction, that $\mu_t (\alpha , p) = \mu_{t_0} (\alpha , p)$, for all $t \geq t_0$ and $\alpha \in \{0,1\}^{\leq k_0}$.

Suppose that $t$ is in case 1.  In the first place, if $t+1$ is also in case 1, then, by condition (2a),
\begin{equation*}\tag{$*$}
\mu_{t+1} \upto \{0,1\}^{\leq k_0} \times \{ 0 , \ldots , p\} = \mu_t \upto \{0,1\}^{\leq k_0} \times \{ 0 , \ldots , p\},
\end{equation*}
since $t \geq t_0 \geq \max \{ k_0 , p\}$.  On the other hand, if $t+1 = \langle k',m',n' \rangle$ is in case 2 and $k_0 < k'$, we have that ($*$) again holds by condition (2b), since $k_0 \leq \min \{ t , k'-1\}$.  Finally, if $t+1 = \langle k',m',n' \rangle$ is in case 2 and $k' \leq k_0$, then ($*$) still holds by (2b), since
\[
\max\{ k_0 , p \} \leq \min \{ t , k'+m' \}.
\]

If $t = \langle k,m,n \rangle$ is in case 2, the arguments are analogous, using (6a) and (6b) above.  For instance, if $k \leq k_0$ and $t+1$ is in case 1, then, by assumption, $k+m \geq \max \{ k_0 , p\}$ and hence condition ($*$) holds by (6a), using the fact that $k_0 \leq \min\{ k+m , t+1 \}$.
\end{proof}

\begin{claim}
If $x \in F_{k_0}$, then
\[
\lim_{s\rightarrow \infty} d_\alpha \big( f(x) \upto s \big) = 2^{-|\alpha|},
\]
for each $\alpha \in \{0,1\}^{\leq k_0}$.
\end{claim}

\begin{proof}
Observe that $y_t \rightarrow f(x)$, as $t \rightarrow \infty$.  The functions $\mu_t \upto \{0,1\}^{\leq k_0} \times \NN$ also form a (pointwise) convergent sequence, by the previous claim.  Fixing $\alpha \in \{0,1\}^{\leq k_0}$, it follows that the sequence $\big( d_\alpha \big( f(x) \upto s \big) \big)_{s \in \NN}$ is Cauchy and therefore convergent.  By conditions (1), (3) and (4) above, for each $\varepsilon > 0$, there are infinitely many $s \in \NN$ such that 
\[
\left| d_\alpha \big( f(x) \upto s \big) - 2^{-|\alpha|} \right| < \varepsilon.
\]
It follows that $\lim_{s \rightarrow \infty} d_\alpha \big( f(x) \upto s \big) = 2^{-|\alpha|}$.
\end{proof}

From the last two claims, we conclude that, if $x \in F_{k_0}$, we have $0 \, . \, f(x) \in N_{k_0}$.  The next claim asserts the converse.

\begin{claim}
If $x \notin F_{k_0}$, then $0 \, . \, f(x) \notin N_{k_0}$.
\end{claim}

\begin{proof}
Assume $x \notin F_{k_0}$ and $m_0 \in \NN$ is such that $x \notin F_{k_0 , m_0 , n}$, for all $n \in \NN$.  For each $n$, let $s_n \in \NN$ be least such that $[ x\upto s_n ] \cap F_{k_0 , m_0 , n} = \emptyset$.  We consider two distinct cases.

First, suppose that there are infinitely many $n$ such that $s_n > \langle m_0 , n \rangle$.  In this case, there exist $n_0 < n_1 < \ldots$ and $p_0 < p_1 < \ldots$ such that, for each $j$,
\begin{itemize}
\item $p_j > n_j$ and
\item $\langle m_0 , p_j - 1\rangle < s_{n_j} \leq \langle m_0 , p_j \rangle$.
\end{itemize}
Thus, for each $j$, 
\[
[x \upto \langle m_0 , p_j - 1\rangle ] \cap F_{k_0, m_0 , n_j} \neq \emptyset \qquad \& \qquad [x \upto \langle m_0 , p_j \rangle] \cap F_{k_0 , m_0 , n_j} = \emptyset.
\]
It follows that each $t_j = \langle k_0 , m_0 , p_j \rangle$ is in case 2 and, hence, for each $j$, 
\[
d_{\alpha_{k_0}} \big( f(x) \upto |\sigma_{t_j}| \big) < r_{k_0 , m_0} < 2^{-k_0},
\]
by condition (5) above.  Thus, $0 \, . \, f(x) \notin N_{k_0}$.

On the other hand, if $s_n \leq \langle m_0 , n\rangle$, for all but finitely many $n$, we have 
\[
[x \upto \langle m_0 , n \rangle] \cap F_{k_0 , m_0 , n} = \emptyset,
\]
for cofinitely many $n$.  Thus, $\langle k_0 , m_0 , n\rangle$ is in case 2 for cofinitely many $n$ and again $0 \, . \, f(x) \notin N_{k_0}$.
\end{proof}

We conclude that, for each $k \geq 1$ and $x \in \{0,1\}^\NN$, we have $ x \in F_k$ iff $0\, . \, f(x) \in N_k$.  It follows that 
\[
x \in \bigcup_k F_{2k+1} \setminus F_{2k+2} \iff 0\, . \, f(x) \in  \bigcup_k N_{2k+1} \setminus N_{2k+2},
\]
for each $x \in \{0,1\}^\NN$.  This completes the proof of Theorem~\ref{T2}.




\end{document}